\documentclass[11pt]{amsart}
\usepackage{amssymb}
\usepackage{palatino}
\usepackage{tikz-cd}
\input amssym.def

\usepackage{xcolor}
\usepackage{amsmath, amsfonts}
\usepackage{amssymb}
\usepackage{amscd}
\usepackage[mathscr]{eucal}
\usepackage{palatino}
\usepackage{theoremref}
\usepackage{tikz}
\usepackage{float}

\newfont{\cyrr}{wncyr10}

\newcommand{\thmref}[1]{Theorem~\ref{#1}}

\newcommand{\lemref}[1]{Lemma~\ref{#1}}

\newtheorem{thm}{Theorem}

\newtheorem{lem}[thm]{Lemma}

\newcommand{\Z}{{\mathbb Z}}

\newcommand{\SL}{{\rm{SL}}}

\def\({\left(}
\def\){\right)}
\def\[{\left[}
\def\]{\right]}

\def\GL{{\rm GL}}
\def\G{{\rm G}}
\def\SL{{\rm SL}}
\def\S{{\rm S}}

\def\Li{{\rm Li}}

\def\N{\mathbb{N}}

\def\Q{\mathbb{Q}}

\def\F{\mathbb{F}}

\def\cP{\mathcal{P}}

\def\cS{\mathcal{S}}

\def\z{\zeta}

\renewcommand{\mod}{ \text{ mod }}
\newcommand{\ord}{\text{ord}}
\renewcommand{\vec}{\overrightarrow}

\title[Hecke eigenvalues of Ikeda lifts]{On divisibility of Hecke eigenvalues of Ikeda lifts}

\author{Sanoli Gun and Sunil Naik}
\address{Sanoli Gun  \newline
	The Institute of Mathematical Sciences, 
	A CI of Homi Bhabha National Institute, 
	CIT Campus, Taramani, 
	Chennai 600 113, 
	India.}
\email{sanoli@imsc.res.in}

\address{Sunil Naik \newline
	Department of Mathematics,
	Queen's University, Jeffrey Hall, 
	99 University Avenue, 
	Kingston, ON K7L3N6, 
	Canada}
\email{naik.s@queensu.ca}

\begin{document}
	
\hfuzz 5pt



\maketitle

\begin{abstract}
In this article, we estimate the density of the set of primes $p$ 
such that the $p$-th Hecke eigenvalue of an Ikeda lift is divisible 
by a fixed positive integer. 
One of the main ingredients involves the study of abelian subfields 
of fixed fields of the kernel of Galois representations 
attached to elliptic Hecke eigenforms. Further, we study 
the distribution of Fourier coefficients of elliptic Hecke eigenforms 
in arithmetic progressions.
\end{abstract}
	
\section{Introduction and Statements of Results}\label{S1}
Throughout the article, let $k, n$ denote 
even positive integers with $k > n+1$ and 
$\Gamma_n = Sp_n(\Z) \subseteq {\rm GL}_{2n}(\Z)$ 
denote the full Siegel modular group of degree $n$. 
Also let $S_k(\Gamma_n)$ denote the space of Siegel cusp forms 
of weight $k$ and degree $n$ for $\Gamma_n$. 
Let $F \in S_k(\Gamma_n)$ be an Ikeda lift and 
for $m \in \N$, $\lambda_F(m)$ be the Hecke eigenvalue of the 
Hecke operator $T(m)$ with respect to $F$. 
It is well-known that $\lambda_F(m)$'s are real algebraic integers 
(see \cite[Theorem 4.1]{HK}, \cite{SM}).
In a recent article \cite{GN}, the authors proved that 
$\lambda_F(p) > 0 $
for all primes $p$. This generalizes a well-known 
result of Breulmann \cite{SB}. 
Further, if $F$ is an Ikeda lift of a 
normalized Hecke eigenform $f \in S_{2k-n}(\Gamma_1)$, 
then we have (see \cite{GN, Mu})
$$	
\lambda_F(p)  ~=~ \prod_{i=1}^{\frac{n}{2}}  
\( a_f(p) + p^{k-i} + p^{k-n-1+i}\),
$$ 
where $a_f(p)$ denotes the $p$-th Fourier coefficient of $f$.
In this article, we study divisibility properties of 
Hecke eigenvalues of Ikeda lifts. For the sake of simplicity, 
we suppose that $\lambda_F(m)$'s are rational integers.
Before moving further, let us introduce few notations. 
Let $S$ be a set of primes. We say that $S$ has a density, 
say $\delta$, if the limit
$$
\lim_{x \to \infty} \frac{\#\{ p \leq x : p \in S\}}{\pi(x)}
$$
exists and is equal to $\delta$. Here $\pi(x)$ denotes 
the number of primes less than or equal to $x$.
For any integer $d>1$, set
$$
\Pi_F(d) ~=~ \{ p ~:~  p \nmid d 
\text{  and   } 
\lambda_F(p) \equiv 0 ~(\mod d)\}
\phantom{m}\text{and} \phantom{m} 
\pi_F(x,d) ~=~ \#\{ p \leq x ~:~ p \in \Pi_F(d)\}.
$$
In this article, we prove the following theorem.

\begin{thm}\label{thm2}
Let $F \in S_k(\Gamma_n)$ be an Ikeda lift of a 
normalized Hecke eigenform in $S_{2k-n}(\Gamma_1)$ 
with integer Fourier coefficients.
Also let $\ell$ be a prime number and $m$ be a positive integer.
Then the density of the set $\Pi_F(\ell^m)$ exists,
denoted by $\delta_F(\ell^m)$, and we have
\begin{equation*}
\delta_F(\ell) ~=~ \frac{n}{2}  
\(\frac{1}{\ell} + O\(\frac{1}{\ell^2}\) \)
\phantom{mm}\text{and}\phantom{mm}
\delta_F(\ell^m) 
~\ll~
\begin{cases}
\frac{1}{\ell^m} & \text{if $n=2$}, \\
\frac{m^2}{\ell^\frac{3m}{n}} & \text{if $n > 2$}.
\end{cases}
\end{equation*}
Further, there exists a positive constant $c$ such that
\begin{equation*}
\pi_F\(x, \ell^m\) ~=~ \delta_{F}(\ell^m) \pi(x) 
~+~ O\(x \exp\(- (\log x)^{\frac13}\)\)
\end{equation*}
uniformly for $\ell^m \leq (\log x)^{c}$.
Moreover, if we assume that the generalized Riemann hypothesis (GRH) is true, then we have
$$
\pi_F(x, \ell^m) ~=~ \delta_F(\ell^m) \( \pi(x) 
~+~ O\(\ell^{4m} x^{\frac{1}{2}} \log (\ell^m x)\)\)
$$
for $x \geq 2$.
\end{thm}

One of the main ingredients in proving \thmref{thm2} is
to study abelian subfields of the kernel of 
Galois representations attached to elliptic Hecke eigenforms. 
Let $f$ be a normalized Hecke eigen cusp form of weight $k$ 
for $\text{SL}_2(\Z)$ having integer 
Fourier coefficients and $\G  = \text{Gal}\(\overline{\Q}/\Q\)$. 
By the work of Deligne \cite{PD}, 
there exists a continuous representation
\begin{equation*}
\rho_{d,f} ~:~ {\G} ~\rightarrow~ {\GL}_2\(\prod_{\ell | d} \Z_\ell\)
\end{equation*}
which is unramified outside $d$. 
Further, if $p \nmid d$, then 
$$
\text{tr}\rho_{d,f}(\sigma_p) ~=~ a_f(p) 
\phantom{mm}\text{and}\phantom{mm} 
\text{det}\rho_{d,f}(\sigma_p) = p^{k-1},
$$
where $\sigma_p$ is a Frobenius element of $p$ in $\G$.
Denote by $\tilde{\rho}_{d,f}$, 
the reduction of $\rho_{d,f}$ modulo $d$ i.e.
\begin{equation*}
\tilde{\rho}_{d,f} ~:~ 
{\G} \xrightarrow[]{\rho_{d,f}} {\GL}_2\(\prod_{\ell | d} \Z_\ell\)
\twoheadrightarrow {\GL}_2(\Z/ d\Z).
\end{equation*}
Let $K_d$ be the subfield of $\overline{\Q}$ fixed by the kernel 
of $\tilde{\rho}_{d,f}$. 
In this set-up, we prove the following theorem.
\begin{thm}\label{thm4}
Let $\ell$ be a prime number and $m$ be a positive integer.
Also let $A_{\ell^m}$ be the subfield of $\Q\(\z_{\ell^m}\)$ 
of degree $\frac{\varphi\(\ell^m\)}{\(k-1, \varphi(\ell^m)\)}$, 
where $\zeta_{\ell^m} = e^{\frac{2\pi i}{\ell^m}}$.
Then we have 
$$
A_{\ell^m} ~\subseteq~ K_{\ell^m} \cap \Q\(\z_{\ell^m}\)
$$ 
and equality holds if $\ell$ is sufficiently large. 
In particular, we have $\Q\(\z_{2^m}\) \subseteq K_{2^m}$. 
Further, $A_{\ell^m}$ is the maximal abelian subfield of $K_{\ell^m}$ 
if $\ell$ is sufficiently large.
\end{thm}
We also prove the following lemma.		
\begin{lem}\label{LAlm}
Let $\ell$ be a prime number and $m \in \N$. 
Then we have $A_{\ell^m} \subseteq A_{\ell^{m+1}}$. Further,
$A_{\ell^m} = A_{\ell}$ for $m \leq \nu_\ell(k-1)+1$
and $[ A_{\ell^{m+1}} :  A_{\ell^{m}}] = \ell$ 
for $m \geq \nu_\ell(k-1)+1$. 
Here $\nu_\ell$ denotes the $\ell$-adic valuation.
\end{lem}
Applying \thmref{thm4} and \lemref{LAlm}, we prove the following theorem.
\begin{thm}\label{thm3}
Let $f$ be a non-CM normalized Hecke eigenform for $\SL_2(\Z)$ 
with integer Fourier coefficients $\{a_f(r) : r \in \N \}$. 
Also let $m \ge 1$ and $u, v$ be integers such that 
$0 \leq u, v < \ell^m$ and $(u, \ell)=1$. Then the density of the set
$$
\Pi_f(u,v; \ell^m) ~=~ 
\bigg\{p  ~:~ ~p \equiv u ~\(\text{ mod } \ell^m\) 
\text{ and } 
a_f(p) \equiv  v ~\(\text{ mod } \ell^m\) \bigg\}
$$
exists, denoted by $\delta_{u,v}(\ell^m)$, and we have
$$
\delta_{u,v}(\ell) 
~=~ \frac{1}{\ell^2} ~+~ O\(\frac{1}{\ell^3}\)
\phantom{mm}\text{and}\phantom{mm} 
\delta_{u,v}(\ell^m) ~\ll~ \frac{1}{\ell^{2m}}
$$
for $m \in \N$.
Further, if
$$
\pi_f\(x,u,v;\ell^m\)
~=~ \#\{p \leq x ~:~ p \in \Pi_f(u,v; \ell^m) \},
$$
then there exists a positive constant $c$ such that
\begin{equation*}
\pi_f\(x,u,v;\ell^m\)
~=~ \delta_{u,v}(\ell^m) \pi(x) 
~+~ O\(x \exp\(- (\log x)^{\frac13}\)\)
\end{equation*}
uniformly for $\ell^m \leq (\log x)^c$.	
Moreover, if we assume that GRH is true, then  we have
$$
\pi_f(x,u, v; \ell^m) 
~=~ \delta_{u, v}(\ell^m) \(\pi(x) 
~+~ O\(\ell^{4m} x^{\frac{1}{2}} \log(\ell^m x)\)\)
$$
for $x \geq 2 $.
\end{thm}

\section{Prerequisites}
	
\subsection{Prerequisites from Siegel modular forms}\label{SIl}
As before, let $S_k(\Gamma_n)$ be the space of Siegel cusp forms 
of weight $k$ and degree $n$ for $\Gamma_n$. 
See \cite{AAb, BGHZ, EZ, HM} for an introduction to Siegel modular forms.

\subsubsection{Saito-Kurokawa lifts}
Let $f \in S_{2k-2}(\Gamma_1)$ be a normalized elliptic Hecke eigenform
with Fourier coefficients $\{ a_f(m) \}_{m \ge 1}$. 
It was conjectured by Saito and Kurokawa \cite{NK} that 
there exists a Hecke eigenform $F \in S_k(\Gamma_2)$ such that
$$
Z_F(s) ~=~ \zeta(s-k+1) \zeta(s-k+2) L(s, f).
$$
Here 
$$
Z_F(s) ~=~ \zeta(2s-2k+4) \sum_{m=1}^{\infty} \frac{\lambda_F(m)}{m^s}
$$
is the spinor zeta function associated with $F$
and 
$$
L(s, f) = \sum_{m=1}^{\infty} \frac{a_f(m)}{m^s}
$$ 
is the modular $L$-function associated with $f$.
This conjecture was resolved by Maass, Andrianov and Zagier 
(see \cite{AAc, HM1, HM2, HM3, DZ} for further details). 

We have the following expression (see \cite{SB}, \cite[p. 4]{GPS}) relating 
Hecke eigenvalues of $F$ and Fourier coefficients of $f$ at primes $p$ :
$$
\lambda_F(p) ~=~ a_f(p) + p^{k-1} + p^{k-2} .	
$$
	
\subsubsection{Ikeda lifts} 
As before, let us assume that $k> n+1$. A generalization of the Saito-Kurokawa 
lifts to higher degrees was predicted by Duke and Imamo\={g}lu.  
They conjectured that for a normalized elliptic Hecke eigenform 
$f \in S_{2k-n}(\Gamma_1)$,  
there exists a Hecke eigenform $F \in S_k(\Gamma_n)$ such that 
$$
L(s, F ; st) ~=~ \zeta(s) \prod_{i=1}^{n} L(s+k-i, f).
$$
Here $L(s, F ; st)$ denotes the  standard $L$-function associated to $F$ 
(see \cite[p. 221]{BGHZ}). The existence of such a lift was 
proved by Ikeda \cite{TI} 
and this lift $F$ is now called an Ikeda lift of $f$. 
Note that when $n=2$, it coincides with Saito-Kurokawa lift. 
For any prime $p$, the explicit relation between 
the $p$-th Hecke eigenvalues of $f$ and its Ikeda lift $F$ 
(see \cite[Lemma 3]{GN}, \cite{Mu}) is given by
\begin{equation}\label{eqlambdPol}
\lambda_F(p)  ~=~ \prod_{i=1}^{\frac{n}{2}}  
\( a_f(p) + p^{k-i} + p^{k-n-1+i}\).
\end{equation}

\medspace

\subsection{Applications of $\ell$-adic Galois representations 
and Chebotarev density theorem}\label{SladicG}
Let $f$ be a normalized Hecke eigen cusp form of weight $k$ 
for $\text{SL}_2(\Z)$ having rational integer 
Fourier coefficients $\{a_f(m): m \in \N \}$. Also let 
$\rho_{d,f}$ denote the continuous representation as before 
(see Section \ref{S1}) and $\tilde{\rho}_{d,f}$ 
denote the reduction of $\rho_{d,f}$ modulo $d$:
\begin{equation*}
\tilde{\rho}_{d,f} ~:~ 
{\G} \xrightarrow[]{\rho_{d,f}} {\GL}_2\(\prod_{\ell | d} \Z_\ell\)
\twoheadrightarrow {\GL}_2(\Z/ d\Z).
\end{equation*}
Suppose that $H_d$ is the kernel of $\tilde{\rho}_{d,f}$, 
$K_d$ is the subfield of 
$\overline{\Q}$ fixed by $H_d$ and 
$$
{\G}_d = \text{Gal}(K_d/\Q) \cong \text{Im}\(\tilde{\rho}_{d,f}\).
$$ 
Further
suppose that $C_d$ is the subset of $\tilde{\rho}_{d,f}(\G)$ 
consisting of elements of trace zero.
For any prime $p\nmid d$, the condition $a_f(p) \equiv 0 ~(\mod d)$ 
is equivalent to
the fact that $\tilde{\rho}_{d,f}(\sigma_p) \in C_d$, 
where $\sigma_{p}$ is a Frobenius 
element of $p$ in $\G$. Hence 
by the Chebotarev density theorem applied to $K_d / \Q$, we have
\begin{equation*}
\#\{p \leq x ~:~ p\nmid d ~\text{ and }~ a_f(p) \equiv 0 ~(\mod d) \}
~\sim~ 
\frac{|C_d|}{|{\G}_d|}\pi(x) 
\phantom{m}
\text{as } x \to \infty.
\end{equation*}

\medspace

\subsection{The image of Galois representation $\rho_{\ell, f}$}\label{SSimagerhol}
From the works of Carayol \cite{Ca}, 
Momose \cite{Mo}, Ribet \cite{Ri75,Ri85}, 
Serre \cite{Se76} and Swinnerton-Dyer \cite{Sw}, 
it follows that the image of $\rho_{\ell, f}$ 
is open in ${\rm GL}_2(\Z_{\ell})$ and 
\begin{equation*}
\rho_{\ell, f}(\G)
~\subseteq~
\bigg\{ A \in \GL_2(\Z_{\ell}) 
~:~ \det (A) \in (\Z^{\times}_{\ell})^{k-1}\bigg\}
\end{equation*}
and equality holds if $\ell$ is sufficiently large.
Hence for any $m \in \N$, we have
\begin{equation}\label{eqimgrholm}
\tilde{\rho}_{\ell^m, f}({\rm G})
~\subseteq~
\bigg\{A \in {\rm GL}_2\(\Z/\ell^m \Z\) ~:~ 
\det(A) \in \(\(\Z/\ell^m \Z\)^\times\)^{k-1}\bigg\}
\end{equation}
and equality holds if $\ell$ is sufficiently large.
Let $r_{\ell} = \text{gcd}(\ell-1, k-1)$, then we have
\begin{equation}\label{eq|Imgrholm|}
|\tilde{\rho}_{\ell, f}(\G)| 
~=~
\frac{(\ell^2-1)(\ell^2-\ell)}{r_\ell} 
\phantom{mm}\text{and}\phantom{mm} 	
|\tilde{\rho}_{\ell^m, f}(\G)| 
~=~
\ell^{4(m-1)} |\tilde{\rho}_{\ell, f}(\G)|
\end{equation}
for all sufficiently large $\ell$ (see \cite[Section 3]{GM}).
We also have  
\begin{equation}\label{eq|rhol|low}
|\tilde{\rho}_{\ell, f}(\G)| 
~\gg~
\ell^4,
\end{equation}
where the implied constant depends only on $f$ 
(see \cite[Lemma 5.4]{MM84}, \cite[Prop. 17]{Se81}).

\medspace

\section{Intersection of the fields $K_{\ell^m}$ and $\Q\(\z_{\ell^m}\)$}\label{SintKLQ}
Let $\ell$ be a prime number and $m$ be a positive integer. 
Also let $\zeta_{\ell^m} = e^{\frac{2\pi i}{\ell^m}}$.
The Galois group of $\Q\(\z_{\ell^m}\)$ over $\Q$ is given by
$$
\text{Gal}\(\Q\(\z_{\ell^m}\) / \Q\) 
~=~ \{\psi_a : 1 \leq a < \ell^m,~ (a, \ell)=1\},
$$
where $\psi_a$ is the automorphism of $\Q\(\z_{\ell^m}\)$ over $\Q$ 
such that $\psi_a \( \z_{\ell^m} \) = \zeta_{\ell^m}^{a}$. 
Thus we have
$
\text{Gal}\(\Q\(\z_{\ell^m}\) / \Q\) 
\cong \(\Z/\ell^m \Z\)^{\times}.
$ 
Set $r_{\ell^m} =\(k-1, \varphi(\ell^m)\)$ 
and let $A_{\ell^m}$ be the subfield of $\Q\(\z_{\ell^m}\)$ 
of degree $\frac{\varphi\(\ell^m\)}{r_{\ell^m}}$.

\subsection{Proof of \thmref{thm4}}\label{SSproofthm4}
From subsection \ref{SladicG}, we have an isomorphism 
\begin{equation}\label{isoGalK}
\G_{\ell^m} 
~=~ \text{Gal}\(K_{\ell^m}/\Q\) 
~\cong~ \tilde{\rho}_{\ell^m,f}\(\G\).
\end{equation}
Let the above isomorphism be denoted by $\varrho_{\ell^m, f}$. Also set
$$
\S_{\ell^m} ~=~ 
\varrho_{\ell^m, f}^{-1} 
\bigg(\tilde{\rho}_{\ell^m,f}\(\G\) 
\cap \SL_2\(\Z/\ell^m \Z\) \bigg)
$$ 
and $B_{\ell^m} = K_{\ell^m}^{\S_{\ell^m}}$ be 
the subfield of $K_{\ell^m}$ fixed by $\S_{\ell^m}$. 
Then we have an isomorphism 
\begin{equation*}
\text{Gal}\(B_{\ell^m} /\Q\) 
~\cong~ \frac{\G_{\ell^m}}{\S_{\ell^m}} 
~\cong~ \frac{\tilde{\rho}_{\ell^m,f}\(\G\)}{\tilde{\rho}_{\ell^m,f}\(\G\) 
\cap \SL_2\(\Z/\ell^m \Z\)}.
\end{equation*}
We know that the determinant map
\begin{equation*}	
\det ~:~ \tilde{\rho}_{\ell^m,f}\(\G\) 
\rightarrow \(\(\Z/\ell^m \Z\)^\times\)^{k-1}
\end{equation*}
is surjective (see \cite[Eq. 9]{Sw}) and kernel of the above map 
is equal to $\tilde{\rho}_{\ell^m,f}\(\G\) \cap \SL_2\(\Z/\ell^m \Z\)$. 
Hence we have
\begin{equation*}
\frac{\tilde{\rho}_{\ell^m,f}\(\G\)}{\tilde{\rho}_{\ell^m,f}\(\G\) 
\cap \SL_2\(\Z/\ell^m \Z\)} 
~\cong~ \(\(\Z/\ell^m \Z\)^\times\)^{k-1}.
\end{equation*}
This gives an isomorphism
\begin{equation}\label{isoGalB}
\text{Gal}\(B_{\ell^m} /\Q\) 
~\cong~ \frac{\G_{\ell^m}}{\S_{\ell^m}} 
~\cong~ \(\(\Z/\ell^m \Z\)^\times\)^{k-1}.
\end{equation}
Hence $B_{\ell^m}$ is an abelian subextension of $K_{\ell^m}$ 
of degree $\frac{\varphi\(\ell^m\)}{r_{\ell^m}}$. 
Since $\ell$ is the unique prime which is ramified in $B_{\ell^m}$, 
we deduce that
\begin{equation}\label{clasB}
B_{\ell^m} ~\subseteq~ \Q\(\z_{\ell^M}\)
\end{equation}
for some $M \in \N$. Now from \eqref{isoGalB}, \eqref{clasB} and 
the uniqueness of $A_{\ell^m}$, 
we conclude that $B_{\ell^m} = A_{\ell^m}$. 
Hence we have $A_{\ell^m} \subseteq  K_{\ell^m} \cap \Q\(\z_{\ell^m}\)$. 
Since $A_{2^m} = \Q\(\z_{2^m}\)$, 
we get $\Q\(\z_{2^m}\) \subseteq K_{2^m}$.
		
\smallskip
		
Now let us assume that $\ell$ is sufficiently large and 
$D_{\ell^m}$ be an abelian subextension of $K_{\ell^m}$ 
containing $A_{\ell^m}$.
Note that $\SL_2\(\Z/\ell^m \Z\) \subseteq \tilde{\rho}_{\ell^m,f}\(\G\)$ 
(see subsection \ref{SSimagerhol}) and the commutator subgroup 
\begin{equation}\label{eq-com}
[\SL_2\(\Z/\ell^m \Z\), \SL_2\(\Z/\ell^m \Z\)] ~=~ \SL_2\(\Z/\ell^m \Z\).
\end{equation}
The quotient
$$
\frac{\G_{\ell^m}}{\text{Gal}\(K_{\ell^m} / D_{\ell^m}\)} 
~\cong~ \text{Gal}\(D_{\ell^m} / \Q\)
$$
is abelian and hence  we get
$$
[\G_{\ell^m}, \G_{\ell^m}] ~\subseteq~ \text{Gal}\(K_{\ell^m} / D_{\ell^m}\).
$$ 
It follows from \eqref{eq-com} that $[\S_{\ell^m}, \S_{\ell^m}] = \S_{\ell^m}$
and hence we get 
$$
\S_{\ell^m} ~=~ [\S_{\ell^m}, \S_{\ell^m}]  
~\subseteq~ [\G_{\ell^m}, \G_{\ell^m}] 
~\subseteq~  \text{Gal}\(K_{\ell^m} / D_{\ell^m}\).
$$
By the Galois correspondence, we get 
$D_{\ell^m} \subseteq A_{\ell^m}$ and 
thus $D_{\ell^m} = A_{\ell^m}$. 
Hence we conclude that $A_{\ell^m} = K_{\ell^m} \cap \Q\(\z_{\ell^m}\) $ 
and $A_{\ell^m}$ is the maximal abelian subfield of $K_{\ell^m}$, 
provided $\ell$ is sufficiently large. \qed

\subsection{Proof of \lemref{LAlm}}
First we note that \lemref{LAlm} is clearly true when $\ell=2$, 
since $A_{2^m} = \Q\(\z_{2^m}\)$.
Hence we suppose that $\ell$ is an odd prime.
By uniqueness of subfields of $\Q\(\z_{\ell^{m+1}}\)$ 
of fixed degree, we deduce that $A_{\ell^m} \subseteq A_{\ell^{m+1}}$. 
Then \lemref{LAlm} follows by noting that
$$
[A_{\ell^{m+1}} : A_{\ell^m}] 
~=~ \ell \cdot \frac{r_{\ell^m}}{r_{\ell^{m+1}}} 
~=~ \ell \cdot  \frac{\(\ell^{m-1}, k-1 \) }{\(\ell^{m}, k-1 \) }
~=~
\begin{cases}
& \ell ~~\phantom{m}\text{ if } ~~m-1 \geq \nu_{\ell}(k-1), \\
& 1  ~~\phantom{m} \text{otherwise}.
\end{cases}
$$
\qed

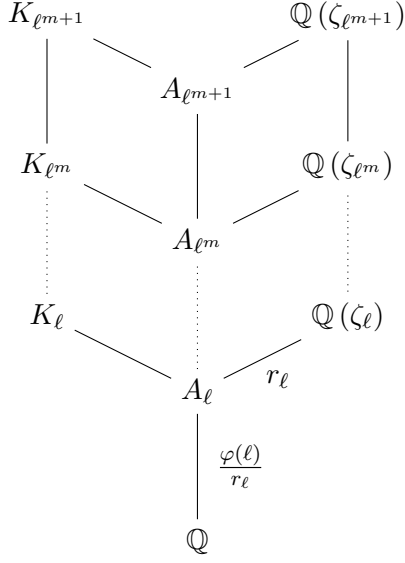
\begin{figure}[H]
\centering
\begin{tikzpicture}
\node (Q1) at (0,0) {$\Q$};
\node (Q2) at (0,2) {$A_\ell$};
\node (Q3) at (0,4) {$A_{\ell^m}$};
\node (Q4) at (0,6) {$A_{\ell^{m+1}}$};
		
\node (Q5) at (2,3) {$\Q\(\z_{\ell}\)$};
\node (Q6) at (2,5) {$\Q\(\z_{\ell^m}\)$};
\node (Q7) at (2,7) {$\Q\(\z_{\ell^{m+1}}\)$};
		
\node (Q8) at (-2,3) {$K_\ell$};
\node (Q9) at (-2,5) {$K_{\ell^m}$};
\node (Q10) at (-2,7) {$K_{\ell^{m+1}}$};
		
\draw (Q1)--(Q2) node [pos=0.5, right,inner sep=0.25cm] {$\frac{\varphi(\ell)}{r_{\ell}}$};
\draw[dotted] (Q2)--(Q3); 
\draw (Q3)--(Q4);
		
\draw (Q2)--(Q5) node [pos=0.7, below,inner sep=0.25cm]{$r_{\ell}$};
\draw (Q3)--(Q6);
\draw (Q4)--(Q7);
		
\draw[dotted] (Q5)--(Q6);
\draw (Q6)--(Q7);
		
\draw (Q2)--(Q8);
\draw (Q3)--(Q9);
\draw (Q4)--(Q10);
		
\draw[dotted] (Q8)--(Q9);
\draw (Q9)--(Q10);
\end{tikzpicture}
\caption{Intersection of $K_{\ell^m}$ and $\Q\(\z_{\ell^m}\)$} \label{fig1}
\end{figure}

\medskip
		
\section{Distribution of Fourier coefficients in Arithmetic progressions}\label{SDistafp}

We will prove the following lemma which is used in the proofs of main theorems.
\begin{lem}\label{lemtrdet}
Let $\ell$ be a prime number and $m$ be a positive integer.	
For any $t \in \Z/\ell^m \Z$ and $d \in \(\Z/\ell^m \Z\)^\times$, let 
$$
E_{\ell^m, t, d} 
~=~ 
\bigg\{ A \in \GL_2\(\Z/\ell^m \Z\) ~:~  \text{tr}(A) = t, ~ \det(A) = d  \bigg\}.
$$
Then we have
\begin{equation*}
\# E_{\ell, t, d} ~=~ \ell^2 +O(\ell) 
\phantom{mm}\text{and}\phantom{mm}	
\# E_{\ell^m, t, d} 
~\ll~  \ell^{2m},
\end{equation*}
where the implied constant is absolute.
\end{lem}
\begin{proof}
If
$$
\begin{pmatrix}
x & y \\
z & w 
\end{pmatrix}
~\in~ E_{\ell^m, t, d},
$$
then we have $x+w = t$ and $xw-yz=d$. This gives
$$
y z ~=~ - (x^2- x t+d).
$$
First consider the case $m=1$.
Set
$$
Z_{1} ~=~ \{ a \in \F_{\ell}  ~:~ a^2-at+d \neq 0 \}
\phantom{mm}\text{and}\phantom{mm}	
Z_{\ell} ~=~ \{ a \in \F_{\ell}  ~:~ a^2-at+d = 0 \}.
$$
Then it can be shown that
$$
\#\Bigg\{\begin{pmatrix}
x & y \\
z & w 
\end{pmatrix}
~\in~ E_{\ell, t, d} ~:~x \in Z_1 \Bigg\}
~=~
|Z_1| \cdot \varphi(\ell)
$$
and 
$$
\#\Bigg\{\begin{pmatrix}
x & y \\
z & w 
\end{pmatrix}
~\in~ E_{\ell, t, d} ~:~ x \in Z_{\ell} \Bigg\}
~=~
|Z_{\ell}| \cdot  \(\varphi(\ell) + \ell\).
$$
Also note that $|Z_{\ell}| \leq 2$.
Thus we get
\begin{equation*}
\begin{split}
\#E_{\ell, t, d} ~=~
|Z_1| \cdot \varphi(\ell) 
~+~ |Z_{\ell}| \cdot  \(\varphi(\ell) + \ell\) 
~=~ \ell \varphi(\ell) +  \ell |Z_\ell|
~=~ \ell^2 + O\(\ell\) .
\end{split}
\end{equation*}
Now assume that $m\geq 2$. We proceed in a similar manner.
For $0 \leq j < m$, set
$$
Z_{\ell^j} ~=~ \bigg\{ a \in \{0, 1, \cdots, \ell^m -1\}  
~:~   \ell^j \mid (a^2- at + d),  ~\ell^{j+1} \nmid (a^2 - at + d) \bigg\}
$$
and
$$
Z_{\ell^m} ~=~ \bigg\{ a \in \{0, 1, \cdots, \ell^m -1\}  
~:~  \ell^m \mid (a^2  - at + d) \bigg\}.
$$
Then it can be shown that for $0 \leq j < m$,
\begin{equation}\label{eqElmZ}
\#\Bigg\{\begin{pmatrix}
x & y \\
z & w 
\end{pmatrix}
~\in~ E_{\ell^m, t, d} ~:~ 
[x]=[a] \text{ for some } a \in Z_{\ell^j} \Bigg\}
~=~ |Z_{\ell^j}| \cdot (j+1)\varphi(\ell^m)
\end{equation}
and
\begin{equation}\label{eqElmZ2}
\#\Bigg\{\begin{pmatrix}
x & y \\
z & w 
\end{pmatrix}
~\in~ E_{\ell^m, t, d} 
~:~ [x]=[a] \text{ for some } a \in Z_{\ell^m} \Bigg\}
~=~ |Z_{\ell^m}| \cdot \(m \varphi(\ell^m)+ \ell^m\).
\end{equation}
Further,
\begin{equation}\label{eqZlj}
|Z_{\ell^j}| ~\leq ~ 16 ~\ell^{m-\frac{j}{2}}
\end{equation}
for any $0 \leq j \leq m$. Hence we get
\begin{equation*}
\begin{split}
\#E_{\ell^m, t, d} 
&~=~ |Z_{\ell^m}| (m \varphi(\ell^m)+ \ell^m) +
\sum_{j=0}^{m-1} (j+1) |Z_{\ell^j}| \varphi(\ell^m)  \\
& ~\ll~ \ell^{m/2} (\ell^m + m \varphi(\ell^m)) 
+  \varphi(\ell^m) \sum_{j=0}^{m-1} (j+1) \ell^{m-j/2} 
~\ll~ \ell^{2m}.
\end{split}
\end{equation*}
This completes the proof of \lemref{lemtrdet}.
\end{proof}

\medspace

\subsection{Proof of \thmref{thm3}}
Let $m \in \N$ and $u, v$ be integers such that 
$0 \leq u, v < \ell^m$ and $(u, \ell)=1$. 
Also let $\Pi_f(u,v; \ell^m)$ be as before
and $\tilde{\rho}_{\ell^m, f}$ be as in 
subsection \ref{SladicG} and $C_{\ell^m, v}$ be the set of elements 
of trace $v$ in $\tilde{\rho}_{\ell^m, f}(\G)$. 
For $p \nmid \ell$, 
the condition $a_f(p) \equiv  v ~\(\text{ mod } \ell^m\)$ 
is equivalent to the fact that 
$\tilde{\rho}_{\ell^m, f}(\sigma_{p}') \in C_{\ell^m, v}$, 
where $\sigma_{p}'$ is a Frobenius element of $p$ in $\G$. 
Also for $p \nmid \ell$, 
$p \equiv u ~\(\text{ mod } \ell^m\)$ is equivalent to 
$\sigma_{p}'' = \psi_u$, where $\sigma_{p}''$ is the
Frobenius element of $p$ in the Galois group of $\Q(\z_{\ell^m})$
over $\Q$. Here $ \psi_u$ is as defined in Section \ref{SintKLQ}.

Let $L_{\ell^m} = K_{\ell^m} \Q\(\z_{\ell^m}\)$ 
be the compositum of the fields $K_{\ell^m}$ and  
$\Q\(\z_{\ell^m}\)$. Also let 
$\tilde{A}_{\ell^m} =  K_{\ell^m} \cap \Q\(\z_{\ell^m}\)$. 
Note that $L_{\ell^m}$ is a Galois extension of $\Q$ 
of degree equal to
$$
[L_{\ell^m} : \Q] 
~=~ \frac{[K_{\ell^m} : \Q] 
[\Q\(\z_{\ell^m}\) : \Q]}{[\tilde{A}_{\ell^m} : \Q]}.
$$
From \thmref{thm4}, we have $A_{\ell^m} \subseteq \tilde{A}_{\ell^m}$ 
and equality holds if $\ell$ is sufficiently large.
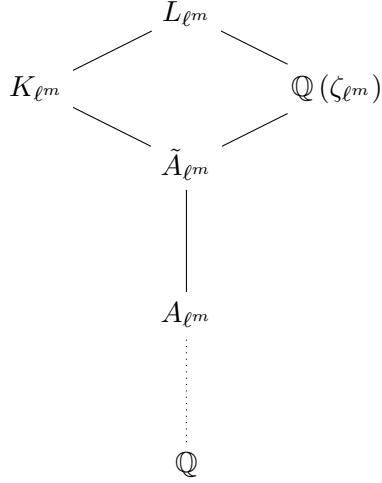
\begin{figure}[H]
\centering
\begin{tikzpicture}
\node (Q1) at (0,0) {$\Q$};
\node (Q2) at (0,2) {$A_{\ell^m}$};
\node (Q3) at (0,4) {$\tilde{A}_{\ell^m}$};
\node (Q4) at (0,6) {$L_{\ell^{m}}$};
\node (Q5) at (-2,5) {$K_{\ell^m}$};
\node (Q6) at (2,5) {$\Q\(\z_{\ell^m}\)$};
		
\draw[dotted] (Q1)--(Q2);
\draw (Q2)--(Q3);
\draw (Q3)--(Q5);
\draw (Q5)--(Q4);
\draw (Q3)--(Q6);
\draw (Q6)--(Q4);
\end{tikzpicture}
\caption{Compositum of $K_{\ell^m}$ 
and $\Q\(\z_{\ell^m}\)$} \label{fig2}
\end{figure}
Further, we have an injective map
\begin{eqnarray*}
i_{\ell^m} ~:~ \text{Gal}\(L_{\ell^m}/ \Q\) 
~&\longrightarrow&~  \text{Gal}\(K_{\ell^m}/ \Q\) 
\times  \text{Gal}\(\Q\(\z_{\ell^m}\)/ \Q\) \\
\sigma ~&\longmapsto&~ 
\(\sigma|_{K_{\ell^m}}, \sigma|_{\Q\(\z_{\ell^m}\)}\)
\end{eqnarray*}
and the image of $i_{\ell^m}$ is given by
$$
\text{Im}\(i_{\ell^m}\) 
~=~ \bigg\{(\sigma_1, \sigma_2) \in \text{Gal}\(K_{\ell^m}/ \Q\)
\times  
\text{Gal}\(\Q\(\z_{\ell^m}\)/ \Q\) 
~:~ \sigma_1|_{\tilde{A}_{\ell^m}} 
= \sigma_2|_{\tilde{A}_{\ell^m}} \bigg\}.
$$
Let $\tilde{C}_{\ell^m, v} 
= \varrho_{\ell^m, f}^{-1}\(C_{\ell^m, v}\)$. 
For $p \nmid \ell$, we have
$$ 
p ~\equiv~ u ~\(\text{ mod } \ell^m\) 
\phantom{m}\text{and}\phantom{m}  
a_f(p) ~\equiv~ v ~\(\text{ mod } \ell^m\) 
\phantom{mm}\iff\phantom{mm}
i_{\ell^m}\(\sigma_{p}\) ~\in~ 
\tilde{C}_{\ell^m, v} \times \{\psi_u\},
$$
where $\sigma_{p}$ is a Frobenius element of $p$ in 
the Galois group of $L_{\ell^m}$ over $\Q$. 
Let 
$$
\tilde{C}_{\ell^m, v, u} 
~=~ \bigg\{ \sigma  \in \tilde{C}_{\ell^m, v} 
~:~ \sigma|_{\tilde{A}_{\ell^m}} 
= \psi_u|_{\tilde{A}_{\ell^m}} \bigg\}
\phantom{mm}\text{and}\phantom{mm}
\delta_{u,v}(\ell^m) 
~=~ \frac{|\tilde{C}_{\ell^m, v, u}|}{[L_{\ell^m} : \Q]}.
$$
By Chebotarev density theorem (see subsection \ref{SladicG}), 
we deduce that
\begin{equation}\label{eqasymppifuv}
\pi_f\(x,u,v;\ell^m\) 
~=~ \#\{p \leq x ~:~ p \in \Pi_f(u,v; \ell^m)\}
~\sim~ \frac{|\tilde{C}_{\ell^m, v, u}|}{[L_{\ell^m} : \Q]} \pi(x) 
\quad \text{ as }~~ x \to \infty.
\end{equation}
Set
$$
\delta_{u,v}(\ell^m)  
~=~ \frac{|\tilde{C}_{\ell^m, v, u}|}{[L_{\ell^m} : \Q]}.
$$
Then \eqref{eqasymppifuv} shows that the density of the set 
$\Pi_f(u,v; \ell^m)$ exists and is equal to $\delta_{u,v}(\ell^m)$.

We partition $C_{\ell^m, v}$ as 
\begin{equation}\label{eqpartCv}
C_{\ell^m, v} 
~=~ \bigsqcup_{d \in \(\(\Z/\ell^m \Z\)^\times\)^{k-1}} 
D_{\ell^m,v,d},
\end{equation}
where $D_{\ell^m,v,d}$ is the set of elements of  
$C_{\ell^m,v}$ of determinant $d$.
Note that if $\ell$ is sufficiently large, 
then $D_{\ell^m,v,d} \neq \emptyset$ for each 
$d \in \(\(\Z/\ell^m \Z\)^\times\)^{k-1}$, 
since
\begin{equation*}
\begin{pmatrix}
& v & -1 \\
& d & 0\\
\end{pmatrix}
~\in~ D_{\ell^m,v,d}.
\end{equation*}
From \eqref{eqpartCv}, we get a partition of $\tilde{C}_{\ell^m, v}$
as 
\begin{equation*}
\tilde{C}_{\ell^m, v} 
~=~ \bigsqcup_{d \in \(\(\Z/\ell^m \Z\)^\times\)^{k-1}} 
\tilde{D}_{\ell^m,v,d},
\end{equation*}
where $\tilde{D}_{\ell^m,v,d} 
= \varrho_{\ell^m, f}^{-1}\(D_{\ell^m,v,d}\)$.
From \thmref{thm4}, we get
\begin{equation}\label{eqCtildeleq}
\tilde{C}_{\ell^m, v, u} 
~=~ \bigg\{ \sigma  \in \tilde{C}_{\ell^m, v} 
~:~ \sigma|_{\tilde{A}_{\ell^m}} 
= \psi_u|_{\tilde{A}_{\ell^m}} \bigg\}
~\subseteq  
\bigg\{ \sigma  \in \tilde{C}_{\ell^m, v} ~:~ \sigma|_{A_{\ell^m}} 
= \psi_u|_{A_{\ell^m}} \bigg\}
\end{equation}
and equality holds if $\ell$ is sufficiently large.
Recall that (see subsection \ref{SSproofthm4})
$$
\text{Gal}\(K_{\ell^m} / A_{\ell^m}\) 
~=~ {\rm S}_{\ell^m} ~=~ \varrho_{\ell^m, f}^{-1}
\bigg( \tilde{\rho}_{\ell^m,f}\(\G\) \cap \SL_2\(\Z/\ell^m \Z\)\bigg) .
$$
Thus for any $d \in \(\(\Z/\ell^m \Z\)^\times\)^{k-1}$ 
and $\sigma_1, \sigma_2 \in \tilde{D}_{\ell^m,v,d}$, we have 
$$
\sigma_1|_{A_{\ell^m}} 
~=~ \sigma_2|_{A_{\ell^m}}
$$
and for $d' \neq d \in \(\(\Z/\ell^m \Z\)^\times\)^{k-1}$ 
and $\sigma' \in \tilde{D}_{\ell^m,v,d'}$, we have
$$
\sigma_1|_{A_{\ell^m}} 
~\neq~ \sigma'|_{A_{\ell^m}}.
$$
This implies that 
$$
\#\{\sigma|_{A_{\ell^m}} 
~:~ \sigma \in \tilde{C}_{\ell^m, v}\} 
~=~ 
\#\(\(\Z/\ell^m \Z\)^\times\)^{k-1}.
$$
Since $[A_{\ell^m} : \Q] = \#\(\(\Z/\ell^m \Z\)^\times\)^{k-1}$, we get
$$
\text{Gal}\(A_{\ell^m} / \Q\) 
~=~\bigg\{\sigma|_{A_{\ell^m}} ~:~ \sigma \in \tilde{C}_{\ell^m, v}\bigg\}.
$$
Note that $\psi_u|_{A_{\ell^m}} \in \text{Gal}\(A_{\ell^m} / \Q\)$. 
Thus there exists a unique $d_u \in \(\(\Z/\ell^m \Z\)^\times\)^{k-1}$ 
such that
\begin{equation*}
\bigg\{ \sigma  \in \tilde{C}_{\ell^m, v} ~:~ \sigma|_{A_{\ell^m}} 
= \psi_u|_{A_{\ell^m}} \bigg\} 
~=~ \tilde{D}_{\ell^m,v,d_u}.
\end{equation*}
Hence from \eqref{eqCtildeleq}, we conclude that
\begin{equation*}
\#\tilde{C}_{\ell^m, v, u}
~\leq~
\#\tilde{D}_{\ell^m, v, d_u}
\end{equation*}
and equality holds if $\ell$ is sufficiently large.
Also, we have
\begin{equation*}
\tilde{D}_{\ell^m,v, d_u} ~\subseteq~
E_{\ell^m, v, d_u}
\end{equation*}
and equality holds if $\ell$ is sufficiently large 
(see subsection \ref{SSimagerhol}).
Thus by \lemref{lemtrdet}, we have
\begin{equation}\label{eqCtildeest}
\#\tilde{C}_{\ell, v, u} ~=~ \ell^2 + O\(\ell\)
\phantom{mm}\text{and}\phantom{mm}
\#\tilde{C}_{\ell^m, v, u} ~\ll~ \ell^{2m}.
\end{equation}
Also, we have
$$
[L_{\ell^m} : \Q] 
~=~ \frac{[K_{\ell^m} : \Q] 
[\Q\(\z_{\ell^m}\) : \Q]}{[\tilde{A}_{\ell^m} : \Q]}
~=~ \frac{|{\rm G}_{\ell^m}| 
\cdot \varphi(\ell^m)}{[A_{\ell^m} : \Q] 
\cdot [ \tilde{A}_{\ell^m} : A_{\ell^m}]}.
$$
Note that 
$$
[ \tilde{A}_{\ell^m} : A_{\ell^m}] 
~\leq~ r_{\ell^m} ~\leq~ k-1
$$
and $[ \tilde{A}_{\ell^m} : A_{\ell^m}] = 1$ if $\ell$ is sufficiently large.
Applying \eqref{eq|Imgrholm|} and the expression for degree of $A_{\ell^m}$ 
(see subsection \ref{SintKLQ}), 
we conclude that
\begin{equation}\label{eqLlmdegree}
[L_{\ell} : \Q] ~=~ \ell^4 + O\(\ell^3\)
\phantom{mm}\text{and}\phantom{mm}
[L_{\ell^m} : \Q] ~\gg~ \ell^{4m}.
\end{equation}
From \eqref{eqCtildeest} and \eqref{eqLlmdegree}, we deduce that
$$
\delta_{u,v}(\ell)  
~=~ \frac{1}{\ell^2} + O\(\frac{1}{\ell^3}\)
\phantom{mm}\text{and}\phantom{mm} 
\delta_{u,v}(\ell^m) \ll \frac{1}{\ell^{2m}}
$$
for $m \in \N$.

Let $\zeta_{L_{\ell^m}}(s)$ be the Dedekind zeta function 
of $L_{\ell^m}$ and $D_{\ell^m}$ be the absolute value of the discriminant 
of $L_{\ell^m}$. Set $n_{\ell^m} = [L_{\ell^m} : \Q]$. 
We have already proved that $n_{\ell^m} \asymp \ell^{4m}$ and 
by a result of Serre \cite[Prop. 6, p. 130]{Se81}, we get
\begin{eqnarray}\label{equpDisc}
\log D_{\ell^m} ~\ll~ n_{\ell^m} \log n_{\ell^m} ~\ll~ \ell^{5m}.
\end{eqnarray}
By applying a result of Lagarias and Odlyzko \cite[Theorem 1.3]{LO}, we get
\begin{eqnarray}\label{eqChebzetalm}
\pi_f\(x,u,v;\ell^m\) ~=~ \delta_{u,v}(\ell^m) \pi(x) ~+~ 
O\(\Li(x^{\beta})\)+O\(x \exp\(-c_1 \sqrt{\frac{\log x}{n_{\ell^m}}}\)\),
\end{eqnarray}
where $\beta$ is a Seigel zero of $\zeta_{L_{\ell^m}}(s)$ (if it exists) and $c_1> 0$ is an absolute constant. By Stark's bound \cite[p. 148]{St}, there exists an absolute constant $c_2> 0$ such that
$$
\beta ~<~ \max\bigg\{ 1 - \frac{1}{4 \log D_{\ell^m}}, ~ 
1 - \frac{c_2}{{D_{\ell^m}}^{1/n_{\ell^m}}}\bigg\}.
$$
Let $c> 0$ be a sufficiently small absolute constant and suppose that $\ell^m \leq (\log x)^c$. Then using the fact that $n_{\ell^m} \asymp \ell^{4m}$ and \eqref{equpDisc}, we deduce that
$$
\beta < 1 - \frac{1}{(\log x)^{\frac{2}{3}}}.
$$
From \eqref{eqChebzetalm}, we deduce that
\begin{equation}\label{eqpifuvLO}
\pi_f\(x,u,v;\ell^m\) ~=~ \delta_{u,v}(\ell^m) \pi(x) 
~+~ O\(x \exp\(- (\log x)^{\frac13} \)\)
\end{equation}
uniformly for $\ell^m \leq (\log x)^c$. 
Further, if we assume GRH for $\zeta_{L_{\ell^m}}(s)$, then by applying a result of 
Lagarias and Odlyzko \cite[Theorem 1.1]{LO} (also see \cite[Lemma 5.3]{MM84}), 
we deduce that for $x \geq 2$,
\begin{equation}\label{eqpifuvLOGRH}
\pi_f(x,u, v; \ell^m) 
= \delta_{u, v}(\ell^m) \(\pi(x) 
~+~ O\(\ell^{4m} x^{\frac12} \log(\ell^m x)\)\).
\end{equation}
This completes the proof of \thmref{thm3}. \qed

\medspace

\subsection{Proof of \thmref{thm2}}
Let the notations be as before and $g_u(x)$ be the polynomial given by 
$$
g_u(x) ~=~ \prod_{i=1}^{\frac{n}{2}}  
\( x + u^{k-i} + u^{k-n-1+i}\).
$$
Then from \eqref{eqlambdPol}, we have
$$
\lambda_F(p) ~=~ g_p(a_f(p))
$$
for all primes $p$.
Thus we get
\begin{equation}\label{eqpiFinf}
\begin{split}
\pi_F(x, \ell^m) 
&~=~ \#\{ p \leq x ~:~ p \nmid \ell ,~
\lambda_F(p) \equiv 0 ~(\mod \ell^m)\} \\
&~=~ \sum_{u=1 \atop (u, \ell)=1}^{\ell^m-1} 
\#\bigg\{ p \leq x ~:~ p \equiv u ~(\mod \ell^m),~ 
g_u(a_f(p)) \equiv 0 ~(\mod \ell^m)\bigg\}\\
&~=~
\sum_{u=1 \atop (u, \ell)=1}^{\ell^m-1} 
\sum_{\substack{w=0 \\ g_u(w) \equiv 0 (\mod \ell^m)}}^{\ell^m-1} 
\#\bigg\{ p \leq x ~:~ p \equiv u ~(\mod \ell^m),~ 
a_f(p) \equiv w ~(\mod \ell^m)\bigg\}\\
&~=~
\sum_{u=1 \atop (u, \ell)=1}^{\ell^m-1} 
\sum_{\substack{w=0 \\ g_u(w) \equiv 0 (\mod \ell^m)}}^{\ell^m-1} 
\pi_f(x, u, w; \ell^m)
\end{split}
\end{equation}
From \eqref{eqasymppifuv}, we get
$$
\lim_{x \to \infty} \frac{\pi_F(x, \ell^m)}{\pi(x)} ~=~ 
\sum_{u=1 \atop (u, \ell)=1}^{\ell^m-1} 
\sum_{\substack{w=0 \\ g_u(w) \equiv 0 (\mod \ell^m)}}^{\ell^m-1} 
 \delta_{u,w}(\ell^m).
$$
We set
$$
\delta_F(\ell^m) 
~=~\sum_{u=1 \atop (u, \ell)=1}^{\ell^m-1} 
\sum_{\substack{w=0 \\ g_u(w) \equiv 0 (\mod \ell^m)}}^{\ell^m-1}  
\delta_{u,w}(\ell^m).
$$
From \thmref{thm3}, we have
\begin{equation}\label{eqNg}
\delta_F(\ell) 
~=~ \sum_{u=1}^{\ell-1} \sum_{\substack{w=0 \\ 
g_u(w) \equiv 0 (\mod \ell)}}^{\ell-1} 
\frac{1}{\ell^2} \(1+ O\(\frac{1}{\ell}\)\) 
~=~ \(\sum_{u=1}^{\ell-1} N_{g_u}(\ell) \) 
\cdot  \frac{1}{\ell^2} \(1+ O\(\frac{1}{\ell}\)\) 
\end{equation}
and 
\begin{equation}\label{eqNgu}
\delta_F(\ell^m) 
~\ll~ \(\sum_{u=1 \atop (u, \ell)=1}^{\ell^m-1} N_{g_u}(\ell^m) \) 
\cdot \frac{1}{\ell^{2m}}
\end{equation}
where
$$
N_{g_u}(\ell^m) 
~=~ \#\{ 0 \leq w  < \ell^m : g_u(w)  \equiv 0 ~( \mod \ell^m) \}.
$$
We have
$$
g_u(x) ~=~ \prod_{i=1}^{\frac{n}{2}} (x - \gamma_{u,i}),
$$
where $\gamma_{u,i} = -u^{k-i} - u^{k-n-1+i}$. 
Thus
$$
g_u(w) ~\equiv~ 0 ~(\mod \ell) ~~\implies~~ 
w ~\equiv~ \gamma_{u,i} ~(\mod \ell)
$$
for some $1 \leq i \leq n/2$. Hence we get 
$$
N_{g_u}(\ell) ~\leq~ \frac{n}{2}
$$
and equality holds if
$
\gamma_{u,i} ~\not\equiv~ \gamma_{u,j} (\mod \ell)
$ 
for any $i \neq j$. Suppose 
$
\gamma_{u,i} ~\equiv~ \gamma_{u,j} (\mod \ell)
$ 
for some $j < i$, then we have
$$
(u^{i-j}-1)(u^{n+1-i-j} -1 ) ~\equiv~ 0 ~(\mod \ell).
$$
This implies that 
$$u^{r} ~\equiv~ 1 ~(\mod \ell) 
\phantom{m}\text{ for some } 1 \leq r \leq n.
$$
Since
$$
\#\bigg\{ u \in \Z/\ell\Z ~:~ \ord_{\ell}(u) = r \bigg\} 
~=~ \varphi(r),
$$
we get
$$
\#\bigg\{ u \in \Z/\ell\Z ~:~ \ord_{\ell}(u) \leq n \bigg\} 
~\leq~ n^2.
$$
Hence we get
\begin{equation}
\begin{split}
\sum_{u=1}^{\ell-1} N_{g_u}(\ell)  
&~=~
\sum_{u=1 \atop \ord_{\ell}(u) > n}^{\ell-1} N_{g_u}(\ell) 
~+~
\sum_{u=1 \atop \ord_{\ell}(u) \leq n}^{\ell-1} N_{g_u}(\ell) \\
&~=~ \sum_{u=1 \atop \ord_{\ell}(u) > n}^{\ell-1} \frac{n}{2} 
+ O\(n^3\) 
~=~ \frac{n}{2} \cdot \ell \(1 + O\(\frac{1}{\ell}\)\).
\end{split}
\end{equation}
From \eqref{eqNg}, we get
$$
\delta_F(\ell) 
~=~ \frac{n}{2} \cdot \frac{1}{\ell} \(1 + O\(\frac{1}{\ell}\)\).
$$
Now suppose that $m$ is an integer strictly greater than $1$. We estimate
$$
\sum_{u=1 \atop (u, \ell)=1}^{\ell^m-1} N_{g_u}(\ell^m)
$$
as follows. If $g_u(w) \equiv 0 ~( \mod \ell^m)$, then 
there exists a partition $\vec{t} = (t_1, t_2, \cdots, t_{n/2})$ 
of $m$ (with $t_1 \geq t_2 \geq \cdots \geq t_{n/2}$)  
and a permutation $\sigma$ in the symmetric group ${\rm S}_{n/2}$ such that
\begin{equation}\label{wcong}
w ~\equiv~ \gamma_{u,\sigma(i)} ~(\mod \ell^{t_i}) 
\phantom{mm} 
\forall ~1 \leq i \leq n/2.
\end{equation}
If $n=2$, then $N_{g_u}(\ell^m) = 1$ and hence in this case we conclude that
\begin{equation*}
\delta_F(\ell^m) ~\ll~ \frac{1}{\ell^m}.
\end{equation*}
We suppose that $n \geq 4$.
If $t_2 = 0$, then $w \equiv \gamma_{u, i} ~(\mod \ell^{m})$
for some $1 \leq i \leq n/2$. 
Now suppose that $t_2 \geq 1$. Then from \eqref{wcong}, we get
$$
\gamma_{u,\sigma(2)} ~\equiv~ \gamma_{u,\sigma(1)} ~(\mod \ell^{t_2}).
$$
This implies that
$$
u^r ~\equiv~ 1 ~(\mod \ell^{t_2'})
$$
for some $1 \leq r \leq n$ and $t_2' = \lfloor \frac{t_2+1}{2} \rfloor$.  
Note that 
$$
\#\bigg\{1 \leq u < \ell^m ~:~ (u,\ell)=1, ~\ord_{\ell^{t_2'}}(u) \leq n \bigg\} 
~\leq~ n^2 \ell^{m-t_2'}.
$$
Let 
$$
\cP_{n/2}(m) ~=~ 
\left\{
\vec{s} = (s_1, s_2, \cdots, s_{n/2}) \in \Z_{\geq 0}^{n/2} 
~:~ s_1 \geq s_2 \geq \cdots \geq s_{n/2} \geq 0,~~~
\sum_{i=1}^{n/2} s_i ~=~ m
\right\}
$$
be the set of partitions of $m$ into $n/2$ parts.
For integers $1 \leq t_2 \leq t_1\leq m$ with $t_1+t_2 \leq m$, 
let $\cS(t_1, t_2)$ be the set of partitions 
$\vec{s}\in \cP_{n/2}(m)$ with $(s_1, s_2) = (t_1, t_2)$. 
We set
\begin{equation*}
\begin{split}
N_{g_u, (t_1, t_2)}(\ell^m) 
~=~ 
&\#\biggl\{ 0 \leq w < \ell^m ~:~ 
w \equiv \gamma_{u,\sigma(i)} ~(\mod \ell^{s_i}) 
~\forall~ 1 \leq i \leq n/2 
\text{ and for some } \sigma\in {\rm S}_{n/2}  \\
& \phantom{mm} \text{ and }  
\vec{s} \in \cS(t_1, t_2) \biggr\}.	
\end{split}
\end{equation*}
As before, if $\ord_{\ell^{t_2'}}(u) > n$,  then $N_{g_u,(t_1, t_2)}(\ell^m)=0$. 
Further, we have
\begin{equation}
\begin{split}
N_{g_u,(t_1, t_2)}(\ell^m)  
&~\leq~ \#\bigg\{0 \leq w < \ell^m ~:~ 
w \equiv \gamma_{u,\sigma(1)} ~(\mod \ell^{t_1}) 
\text{ for some } \sigma \in {\rm S}_{n/2} \bigg\} \\
& ~\leq~ \frac{n}{2} ~\ell^{m-t_1}.
\end{split}
\end{equation}
Hence we get
\begin{eqnarray}\label{eqNgut1t2}
\sum_{u=1 \atop (u, \ell)=1}^{\ell^m-1}  N_{g_u}(\ell^m)
&~\leq~&  
\sum_{u=1 \atop (u, \ell)=1}^{\ell^m-1} N_{g_u, (m,0)}(\ell^m) 
~+~ \sum_{u=1 \atop (u, \ell)=1}^{\ell^m-1} 
\sum_{1 \leq t_2 \leq t_1 \leq m \atop  \cS(t_1, t_2) \neq \emptyset} 
N_{g_u, (t_1,t_2)}(\ell^m) \nonumber \\
&~\leq~& 
\frac{n}{2}~ \ell^m ~ +~
\frac{n^3}{2}
 \sum_{1 \leq t_2 \leq t_1 \leq m \atop \cS(t_1, t_2) 
\neq \emptyset} 
 \ell^{m-t_1}  \ell^{m-t_2'} \nonumber \\
&~\ll~&  
\ell^m ~+~ \ell^{2m}\sum_{1 \leq t_2 \leq t_1 \leq m \atop \cS(t_1, t_2)  
\neq \emptyset}  
\frac{1}{\ell^{t_1+t_2'}}.
\end{eqnarray}
We note that 
\begin{equation}\label{eqmint1+t2'}
\min\bigg\{
s_1 + \bigg\lfloor \frac{s_2+1}{2} \bigg\rfloor 
~:~ \vec{s} = (s_1, s_2, \cdots, s_{n/2}) \in \cP_{n/2}(m)\bigg\} 
~\geq~ 
\frac{3m}{n} .
\end{equation}
Further, if 
$$
m ~=~  \frac{n}{2}~q  ~+~ i
$$
for $0 \leq i < n/2$, 
then the minimum of 
$\bigg\{
s_1 + \big\lfloor \frac{s_2+1}{2} \big\rfloor 
~:~ \vec{s} \in \cP_{n/2}(m)\bigg\} $
is attained at
$$
\vec{s} ~=~ (\underbrace{q+1, \cdots, q+1}_{i \text{ times}}, 
\underbrace{q, \cdots, q}_{\(\frac{n}{2}-i\) \text{ times}})  
\in ~\cP_{n/2}(m).
$$
From \eqref{eqNgu}, \eqref{eqNgut1t2} and \eqref{eqmint1+t2'}, 
we conclude that
$$
\delta_F(\ell^m) ~\ll~ \frac{m^2}{\ell^{\frac{3m}{n}}}
$$
provided $n > 2$. The second part of \thmref{thm2} 
follows from \eqref{eqpifuvLO}, \eqref{eqpifuvLOGRH} and \eqref{eqpiFinf}. 
This completes the proof of \thmref{thm2}.   \qed

\medspace

\section*{Acknowledgments}
The authors would like to thank the referee for valuable suggestions.
The first author would like to thank DAE number theory plan project.
The second author would like to acknowledge 
the Institute of Mathematical Sciences (IMSc), India and 
Queen's University, Canada for providing excellent atmosphere to work.

\medspace

\end{document}